\documentclass[10pt,draft,reqno]{amsart}
     \makeatletter
     \def\section{\@startsection{section}{1}%
     \z@{.7\linespacing\@plus\linespacing}{.5\linespacing}
     {\bfseries
     \centering
     }}
     \def\@secnumfont{\bfseries}
     \makeatother
\usepackage{amsmath,amssymb, amsbsy}
\usepackage{subfigure}
\usepackage{graphpap,latexsym,epsf}
\usepackage{color,psfrag}
\usepackage[dvips]{graphicx}
\usepackage{enumerate}
\usepackage{bbm}
\usepackage{relsize}

\newcommand{\F}{{\mathcal F}}
\newcommand{\ieq}{\begin{equation}}
\newcommand{\eeq}{\end{equation}}
\newcommand{\ieqa}{\begin{eqnarray}}
\newcommand{\eeqa}{\end{eqnarray}}
\newcommand{\ieqas}{\begin{eqnarray*}}
\newcommand{\eeqas}{\end{eqnarray*}}
\newcommand{\f}{\hat{f}}

\newcommand{\1}{\mathlarger{\mathlarger{\mathbbm{1}}}}

\newtheorem{theorem}{Theorem}[section]
\newtheorem{lemma}[theorem]{Lemma}

\newtheorem{corollary}[theorem]{Corollary}
\theoremstyle{definition}
\newtheorem{definition}[theorem]{Definition}

\theoremstyle{remark}
\newtheorem{remark}[theorem]{Remark}
\numberwithin{equation}{section}
\begin{document}

\title[Optimal portfolios for different anticipating integrals]{Optimal portfolios for different anticipating integrals under insider information}

\author[Carlos Escudero]{Carlos Escudero*}
\thanks{* This work has been partially supported by the Government of Spain (Ministerio de Ciencia, Innovaci\'on y Universidades) through Project PGC2018-097704-B-I00.}
\address{Carlos Escudero: Departamento de Matem\'aticas Fundamentales, Universidad Nacional de Educaci\'on a Distancia, Spain}
\email{cescudero@mat.uned.es}

\author{Sandra Ranilla-Cortina}
\address{Sandra Ranilla-Cortina: Departamento de An\'alisis Matem\'atico y Matem\'atica Aplicada, Universidad Complutense de Madrid, Spain}
\email{sranilla@ucm.es}



\begin{abstract}
We consider the non-adapted version of a simple problem of portfolio optimization in a financial market that results from the presence of insider information. We analyze it via anticipating stochastic calculus and compare the results obtained by means of the Russo-Vallois forward, the Ayed-Kuo, and the Hitsuda-Skorokhod integrals. We compute the optimal portfolio for each of these cases with the aim of establishing
a comparison between these integrals in order to clarify their potential use in this type of problem.
Our results give a partial indication that, while the forward integral yields a portfolio that is financially meaningful, the Ayed-Kuo and the Hitsuda-Skorokhod integrals do not provide an appropriate investment strategy for this problem. 
\end{abstract}

\maketitle

\section{Introduction}\label{introduction}
Many mathematical models in the applied sciences are
expressed in terms of stochastic differential equations such as
\begin{equation}\label{white}
\frac{dx}{dt}= a(x,t) + b(x,t) \, \xi(t),
\end{equation}
where $\xi(t)$ is a ``white noise''. Of course, this equation cannot be
understood in the sense of the classical differential calculus of Leibniz and Newton.
Instead, the use of stochastic calculus provides a precise meaning to these models. However, the choice of a particular notion of stochastic integration has
generated a debate that has expanded along decades~\cite{mmcc,kampen}.
This debate, which has been particularly intense in the physics literature, has been focused mainly on the choice between the It\^o integral~\cite{ito1,ito2}, which leads to the notation
\begin{equation*}\label{ito}
dx = a(x,t) \, dt + b(x,t) \, dB(t),
\end{equation*}
and the Stratonovich integral~\cite{stratonovich} usually denoted as
\begin{equation*}\label{str}
dx = a(x,t) \, dt + b(x,t) \circ dB(t),
\end{equation*}
where $B(t)$ is a Brownian motion. Also, quite often in the physics literature,
one finds different meanings associated to equation~\eqref{white}~\cite{mmcc}.
The use of different notions of stochastic integration leads, in general, to different dynamics and, when they exist, to different stationary probability distributions~\cite{hl}; but it may also lead to different numbers of solutions~\cite{ce,escudero2}.

Since the preeminent place for this debate on the noise interpretation has been the physics literature, the focus has been put on diffusion processes, perhaps due to the influence of the seminal works by Einstein~\cite{einstein} and Langevin~\cite{langevin}. Herein we move out of even the Markovian setting and, following the steps of~\cite{bastonsescudero,escudero}, we concentrate on stochastic differential equations with non-adapted terms.
Non-adaptedness arises in financial markets concomitantly to the presence of insider traders.
Let us exemplify this with a model that is composed by two assets, the first of which is free of risk such as a bank account
\begin{eqnarray}\nonumber
d S_0 &=& \rho \, S_0 \, d t \\ \nonumber
S_0(0) &=& M_0,
\end{eqnarray}
and the second one is risky such a stock
\begin{eqnarray}\nonumber
d S_1 &=& \mu \, S_1 \, d t + \sigma \, S_1 \, d B(t) \\ \nonumber
S_1(0) &=& M_1.
\end{eqnarray}
This model is characterized by the set of positive parameters
$\{M_0, M_1, \rho, \mu, \sigma\}$, each one of which has an economic significance:
\begin{itemize}
\item $M_0$ is the initial wealth to be invested in the bank account,
\item $M_1$ is the initial wealth to be invested in the stock,
\item $\rho$ is the interest rate of the bank account,
\item $\mu$ is the appreciation rate of the stock,
\item $\sigma$ is the volatility of the stock.
\end{itemize}
Therefore the total initial wealth is $M = M_0+M_1$. Moreover we assume the inequality $\mu > \rho$ that imposes the higher
expected return of the risky investment. We allow only buy-and-hold strategies in which the trader divides the fixed initial amount $M$ between the two assets; that is, long-only strategies are allowed.
The total wealth at time $t$ is
$$
S^{\text{(I)}}(t):=S_0(t)+S_1(t),
$$
that is, the sum of the returns from the bank account and the stock.
Assuming a fixed time horizon $T$ and
using It\^o calculus we can compute the expected final total wealth, which is given by
\begin{eqnarray}\nonumber
\mathbb{E}\left[S^{\text{(I)}}(T)\right]
&=& M_0 \, e^{\rho T} + M_1 \, e^{\mu T}
\\ \nonumber
&=& M \left[ \frac{M_0}{M} \, e^{\rho T} + \frac{M_1}{M} \, e^{\mu T} \right] \\ \nonumber
&=& M \left[ \frac{M_0}{M} \, e^{\rho T} + \frac{M-M_0}{M} \, e^{\mu T} \right].
\end{eqnarray}
The last expression, written in terms of a convex linear combination, shows that this quantity can be maximized by choosing the optimal pair
\begin{eqnarray} \nonumber
M_0 &=& 0 \\ \nonumber
M_1 &=& M-M_0=M.  
\end{eqnarray}
The corresponding maximal expected wealth therefore reads
\begin{equation*}\label{average}
\mathbb{E}\left[S^{\text{(I)}}(T)\right] = M e^{\mu T}.
\end{equation*}
This optimization problem is simple enough to allow for
an analytical approach to the extension we will consider herein. In particular, we will permit random and non-adapted initial conditions that will model the knowledge of an insider trader. Under these conditions, we will derive the optimal portfolio for different notions of stochastic integration. The precise problem is presented in the next section.

\section{Insider trading}\label{insider}

We consider a financial market in which
an insider trader, who possesses information on the future price of a stock, is present.
Precisely we assume the trader knows the value $S_1(T)$ at time $t=0$, what we will implement mathematically as if she knew the value $B(T)$.
Moreover we only consider buy-and-hold strategies for which shorting is not allowed. 
Mathematically this translates into her control of the anticipating initial condition $f(B(T))$, where $f \in L^{\infty}(\mathbb{R})$ and $0 \leq f \leq 1$, for the stock. Therefore, in our two assets market, we find the following two equations that model the insider wealth:
\begin{subequations}
\begin{eqnarray}\label{rode1}
d S_0 &=& \rho \, S_0 \, d t \\ \label{rode2}
S_0(0) &=& M \left(1-f(B(T))\right),
\end{eqnarray}
\end{subequations}
and
\begin{subequations}
\begin{eqnarray}\label{s1}
d S_1 &=& \mu \, S_1 \, d t + \sigma \, S_1 \, d B(t) \\ \label{s10}
S_1(0) &=& M f(B(T)).
\end{eqnarray}
\end{subequations}

This system of equations, as written in~\eqref{s1} and \eqref{s10}, is ill-posed. While problem~\eqref{rode1} and \eqref{rode2} could still be considered a random differential equation, the non-adaptedness of initial condition~\eqref{s10}
implies that equation~\eqref{s1} is ill-posed as an It\^o stochastic differential equation.
There is a way out of this pitfall that consists in changing the notion of stochastic integration from the It\^o integral to one of its generalizations that admit non-adapted integrands.
The following sections analyze three different possibilities: the Russo-Vallois forward, the Ayed-Kuo and the Hitsuda-Skorokhod stochastic integrals. We compute the optimal investment strategy provided by each of these integrals, and subsequently we compare them. We will show that, while any of these anticipating stochastic integrals guarantees the well-posedness of the problem at hand, the financial consequences of each choice might be very different.
The aim of this work is to clarify the suitability of the use of each of these stochastic integrals in this particular portfolio optimization.
We anticipate that the forward integral is the only one among these that provides an optimal investment strategy that takes advantage of the anticipating condition in the financial sense.
Therefore our analysis further supports the use of this integral in a financial context as employed,
for instance, in~\cite{bo,noep,do1,do2,do3,leon,nualart}. It is also worth mentioning that the problem of insider trading has been studied by means of different approaches, notably enlargement of filtrations~\cite{jyc,pk},
but semimartingale treatments are not as general as the use of anticipating stochastic calculus~\cite{noep}.

Let us conclude this section mentioning that this problem can be approached by means of much simpler
mathematical techniques, as it is done in section~\ref{frc}. However, the goal of this work is not to
solve this simplified financial question that, as said, is relatively simple to address. Our objective
is to give a partial indication of the use of three different anticipating stochastic integrals in
finance. For this purpose we need a simple enough problem at least for two reasons. The first is that
the availability of explicit solutions to stochastic differential equations interpreted in either the
Hitsuda-Skorokhod or Ayed-Kuo senses is limited~\cite{noep,hksz}. The second is that we want to maximize
the clarity of our derivations. Once the role of the different anticipating integrals in the context of
finance is clear, one is able to approach much more general questions as done for instance
in~\cite{bo,noep,do1,do2,do3,leon,nualart} with the forward integral. Herein we aim to highlight
some of the contrasts between the uses of these integrals in a financial context, so the nature
of this work is fundamentally methodological.

\section{The Russo-Vallois integral}\label{rvi}
The Russo-Vallois forward integral was introduced by F. Russo and P. Vallois in 1993 in~\cite{russovallois}. This stochastic integral generalizes the It\^o one, in the sense that it allows to integrate anticipating processes, but it produces the same results as the latter when the integrand is adapted~\cite{noep}.
\begin{definition}
A stochastic process $\{\varphi(t), t \in[a,b]\}$ is \textit{forward integrable} (in the weak sense) over $[a,b]$ with respect to Brownian motion $\{B(t), t\in[a,b]\}$ if there exists a stochastic process $\{I(t), t \in [a,b]\}$ such that
\begin{eqnarray}\nonumber
\sup_{t\in[a,b]} \left | \int_a^t \varphi(s) \frac{B(s+\varepsilon)-B(s)}{\varepsilon} ds - I(t) \right| \to 0, \ \ \  \ \mbox{as} \ \varepsilon \to 0^+,
\end{eqnarray}
in probability. In this case, $I(t)$ is the \textit{forward integral} of $\varphi(t)$ with respect to $B(t)$ on $[a,b]$ and we denote
\begin{eqnarray}\nonumber
I(t) &:=& \int_a^t \varphi(s) \, d^- B(s),  \ \ \ \ t \in [a,b].
\end{eqnarray}
\end{definition}
When the choice is the Russo-Vallois integral, we face the initial value problem
\begin{subequations}
\begin{eqnarray}\label{rv1}
d^- S_1 &=& \mu \, S_1 \, dt + \sigma \, S_1 \, d^- B(t) \\ \label{rv2}
S_1(0) &=& M f(B(T)),
\end{eqnarray}
\end{subequations}
where $d^-$ denotes the Russo-Vallois forward stochastic differential, and $f$ was defined in the previous section. It follows directly from the results in~\cite{noep} that this problem possesses a unique solution.
In the next statement, we establish the optimal investment strategy for the insider trader under Russo-Vallois forward integration.
\begin{theorem}\label{mainthrv}
Let $f$ be a function of $B(T)$ such that $f \in L^{\infty}(\mathbb{R})$ and $0 \leq f \leq 1$. The optimal investment strategy under Russo-Vallois integration is
\begin{eqnarray}\nonumber
f(B(T)) &=& \mathlarger{\mathlarger{\mathbbm{1}}}_{ \big \lbrace B(T) > \frac{T}{\sigma} \left(\rho - \mu + \frac{1}{2}\sigma^2\right) \big \rbrace},
\end{eqnarray}
for model~\eqref{rode1}-\eqref{rode2} and~\eqref{rv1}-\eqref{rv2}.
\end{theorem}
\begin{proof}
The Russo-Vallois integral preserves It\^o calculus~\cite{noep} so using this classical stochastic calculus it is possible to solve problem~\eqref{rv1}-\eqref{rv2} explicitly to find
\begin{subequations}
\begin{eqnarray}\nonumber
S_0(t) &=& M \left(1-f(B(T))\right) e^{\rho t} \\ \nonumber
S_1(t) &=& M f(B(T))e^{\left(\mu - \sigma^2 /2 \right)t + \sigma B(t)}.
\end{eqnarray}
\end{subequations}
Our goal is to find a strategy $f$ such that $\mathbb{E}\left[M^{\text{(RV)}}(T)\right]$,
with $M^{\text{(RV)}}(T):=S_1(T)+S_0(T)$, is maximized. We compute
\begin{eqnarray}\nonumber
\mathbb{E}\left[M^{\text{(RV)}}(T)\right] &=& \mathbb{E}\left[S_0(T)\right] + \mathbb{E}\left[S_1(T)\right] \\ \nonumber
&=& M \left( 1- \mathbb{E}\left[f(B(T))\right] \right) e^{\rho T} + M \mathbb{E}\left[f(B(T))\right]e^{\sigma B(T)} e^{\left(\mu - \sigma^2/2\right)T} \\ \nonumber
&=& M e^{\rho T} \left( 1 - \int_{-\infty}^{\infty} \frac{1}{\sqrt{2 \pi T}} f(x) e^{-\frac{x^2}{2T}}dx \right) \\ \nonumber
&& + Me^{\left(\mu - \sigma^2 /2 \right)T} \left( \int_{-\infty}^{\infty} \frac{1}{\sqrt{2 \pi T}} f(x) e^{-\frac{x^2}{2T}} e^{\sigma x} dx \right),
\end{eqnarray}
where we have used that $B(T) \sim \mathcal{N}(0,T)$. Now consider
\begin{eqnarray}\nonumber
    \Bar{M}(f(x)) &:=& \frac{\mathbb{E}\left[M^{\text{(RV)}}(T)\right]}{M}.
\end{eqnarray}
Hence, we get
\begin{eqnarray}\nonumber
\Bar{M}(f(x)) &=& e^{\rho T} \left( 1 - \int_{-\infty}^{\infty} \frac{1}{\sqrt{2 \pi T}} f(x) e^{-\frac{x^2}{2T}}dx \right) \\ \nonumber
&& + e^{\left(\mu - \sigma^2 /2 \right)T} \left( \int_{-\infty}^{\infty} \frac{1}{\sqrt{2 \pi T}} f(x) e^{-\frac{x^2}{2T}} e^{\sigma x} dx \right) \\ \nonumber
&=& e^{\rho T} + \int_{-\infty}^{\infty} \frac{1}{\sqrt{2 \pi T}} f(x) e^{-\frac{x^2}{2T}} \left( -e^{\rho T} + e^{\left(\mu -\sigma^2/2\right)T} e^{\sigma x} \right) dx.
\end{eqnarray}
The sign of this integrand is determined by the value of $x$, and it possesses a unique root $x_c$ at
\begin{eqnarray}\nonumber
    x_c &=& \frac{T}{\sigma} \left( \rho - \mu + \frac{\sigma^2}{2} \right ).
\end{eqnarray}
The inequality $x > x_c$ corresponds to
\begin{eqnarray}\nonumber
    B(T) &>& \frac{T}{\sigma} \left( \rho - \mu + \frac{\sigma^2}{2} \right ),
\end{eqnarray}
and to the positivity of the integrand. Clearly, in this interval we should take $f$ as large as possible in order to maximize $\mathbb{E}\left[M^{\text{RV}}(T)\right]$. On the other hand, if $x < x_c$ then the integrand is negative and we should take $f$ as small as possible for the same reason. This, together with the assumption $0 \leq f \leq 1$, renders the optimal investment strategy
\begin{eqnarray}\nonumber
    f(B(T)) &=& \mathlarger{\mathlarger{\mathbbm{1}}}_{\big\lbrace B(T) > \frac{T}{\sigma} \left( \rho - \mu + \frac{\sigma^2}{2} \right ) \big\rbrace}.
\end{eqnarray}
\end{proof}

\begin{remark}
The function $f$ from Theorem \ref{mainthrv} implies that the trader should invest the whole amount $M$ in the bank account or the stock according to the value of $B(T)$, in particular, in the asset which actual value is larger at the maturity time $t=T$ (something that is known by the insider). Hence, this investment strategy not only maximizes the expected value $\mathbb{E}\left[M^{\text{(RV)}}(T)\right]$, but it does also take advantage of the anticipating condition in an intuitive way. Thus, the Russo-Vallois integral works as one would expect from the financial point of view, at least for this formulation of the insider trading problem.
\end{remark}

\begin{remark}
The expected wealth of the Russo-Vallois insider under this optimal strategy was computed in~\cite{escudero} and reads
\begin{eqnarray}\nonumber
\mathbb{E}\left[M^{\text{(RV)}}(T)\right] &=& M \, \Phi \left[ \frac{(\sigma^2 + 2 \rho - 2 \mu)\sqrt{T}}{2 \, \sigma} \right] e^{\rho T}
\\ \nonumber & & + \,
M \, \Phi \left[ \frac{(\sigma^2 - 2 \rho + 2 \mu)\sqrt{T}}{2 \, \sigma} \right] e^{\mu T},
\end{eqnarray}
where
$$
\Phi\,(\cdot)= \frac{1}{\sqrt{2 \pi}} \int_{- \infty}^{\, \cdot} \! e^{-s^2/2} \, ds ,
$$
is the cumulative distribution function of the standard normal distribution.
In the same reference it was proven that
$$
\mathbb{E}\left[S^{\text{(I)}}(T)\right] < \mathbb{E}\left[M^{\text{(RV)}}(T)\right],
$$
a fact that matches well with what one expects from the financial viewpoint.
\end{remark}

\section{The Ayed-Kuo integral}\label{aki}
The Ayed-Kuo integral was introduced in~\cite{akuo1,akuo2}. As the previous theory, it generalizes the It\^o integral to anticipating integrands.
Let us now consider a Brownian motion $\{B(t), t \geq 0\}$ and a filtration $\{\F_t, t \geq 0\}$ such that:
\begin{itemize}
\item[(i)] For all $t \geq 0$, $B(t)$ is $\F_t$-measurable,
\item[(ii)] for all $0 \leq s < t$, $B(t)-B(s)$ is independent of $\F_s$.
\end{itemize}
We also recall the notion of instantly independent stochastic process introduced in~\cite{akuo1}.
\begin{definition}\label{ayedkuo}
A stochastic process $\{\varphi(t)$, $t \in [a,b]\}$ is said to be an \textit{instantly independent stochastic process} with respect to the filtration $\{\F_t, t \in [a,b]\}$, if and only if, $\varphi(t)$ is independent of $\F_t$ for each $t \in [a,b]$.
\end{definition}
For the Ayed-Kuo integral, the integrand is assumed to be a product of an adapted stochastic process with respect to the filtration $\F_t$ and an instantly independent stochastic process. Next, we recall its definition as given in~\cite{akuo1}.
\begin{definition}
Let $\{f(t), t\in [a,b]\}$ be a $\{\F_t\}$-adapted stochastic process and let $\{\varphi(t), t\in [a,b]\}$ be an instantly independent stochastic process with respect to $\{\F_t, t \in [a,b]\}$. The \textit{Ayed-Kuo stochastic integral} of $f(t)\varphi(t)$ is defined by
\begin{eqnarray}\nonumber
\int_a^b f(t) \varphi(t) \, d^*B(t) &: =& \lim_{\left| \Pi_n \right| \to 0} \sum_{i=1}^n f(t_{i-1}) \varphi(t_i) \left(B(t_i)-B(t_{i-1}) \right),
\end{eqnarray}
provided that the limit in probability exists, where $\Pi = \lbrace a = t_0, t_1, t_2, ..., t_n = b \rbrace $ is a partition of the interval $[a,b]$ and $\left| \Pi_n \right| = \max _{1 \leq i \leq n} \left(t_i - t_{i-1}\right)$.
\end{definition}

This definition was extended in~\cite{hksz}, and this extension is the one we are going to use herein.
It reads:

\begin{definition}\label{extak}
Consider a sequence $\lbrace \Phi_n(t) \rbrace _{n=1}^{\infty}$ of stochastic processes of the form
\begin{eqnarray}\label{defkuo}
\Phi_n(t) &: =& \sum_{i=1}^k f_{i}(t) \varphi_{i}(t),  \ \ \ \ a \leq t \leq b,
\end{eqnarray}
where $f_i(t)$ are $\{\F_t\}$-adapted continuous stochastic processes and $\varphi_i(t)$ are continuous stochastic processes being instantly independent of $\{\F_t\}$, and $k \in \mathbb{N}$.
Suppose $\Phi(t)$ is a stochastic process satisfying the conditions:
\begin{itemize}
\item[(a)] $\int_a^b \left| \Phi(t) - \Phi_n(t) \right|^2 dt \ \to \ 0 \ \mbox{almost surely,}$
\item[(b)] $\int_a^b \Phi_n(t) d^{*}B(t) \ \  \mbox{converges in probability,}$
\end{itemize}
as $n \to \infty$. Then the \textit{stochastic integral} of $\Phi(t)$ is defined by
\begin{eqnarray}\nonumber
\int_a^b \Phi(t) d^{*}B(t) &: =& \lim_{n \to \infty} \int_a^b \Phi_n(t) d^{*}B(t),  \ \ \ \ \text{in probability}.
\end{eqnarray}
\end{definition}

\begin{remark}
The integral
$$
\int_a^b \Phi_n(t) d^{*}B(t)
$$
is defined by linearity from Definition~\ref{ayedkuo}; its consistency was proven in~\cite{hksz}.
\end{remark}

For the Ayed-Kuo integral we denote the initial value problem as
\begin{subequations}
\begin{eqnarray} \label{ak1}
d^* S_1 &=& \mu \, S_1 \, dt + \sigma \, S_1 \, d^* B(t) \\ \label{ak2}
S_1(0) &=& M f(B(T)),
\end{eqnarray}
\end{subequations}
where $d^*$ denotes the Ayed-Kuo stochastic differential, and $f$ is as before.
This anticipating stochastic differential equation, as the previous case, is known to have
a unique and explicitly computable solution at least when $f \in \mathcal{C}(\mathbb{R})$~\cite[Theorem 4.8]{hksz}.
Now we need to extend this result to $f \in L^\infty(\mathbb{R})$, and to this end we introduce the following
auxiliary result.

\begin{lemma}\label{lemmappr}
Let $\Upsilon(t)$ be a stochastic process that satisfies the conditions:
\begin{itemize}
\item[(a)] $\int_a^b \left| \Upsilon(t) - \Upsilon_n(t) \right|^2 dt \ \to \ 0 \ \mbox{almost surely,}$
\item[(b)] $\int_a^b \Upsilon_n(t) d^{*}B(t) \ \  \mbox{converges in probability,}$
\end{itemize}
as $n \to \infty$, where $\lbrace \Upsilon_n(t) \rbrace _{n=1}^{\infty}$ is a sequence of
Ayed-Kuo integrable (in the sense of Definition~\ref{extak}) stochastic processes.
Then $\Upsilon(t)$ is an Ayed-Kuo integrable process and
\begin{eqnarray}\nonumber
\int_a^b \Upsilon(t) d^{*}B(t) &=& \lim_{n \to \infty} \int_a^b \Upsilon_n(t) d^{*}B(t),  \ \ \ \ \text{in probability}.
\end{eqnarray}
\end{lemma}

\begin{proof}
Let $\lbrace \Phi_{m}^{(n)}(t) \rbrace _{m=1}^{\infty}$ be an approximating sequence
of $\Upsilon_n(t)$ as in Definition~\ref{extak}.
Consider for $\varepsilon>0$ the inclusion of events
\begin{eqnarray}\nonumber
&& \left\{ \left| \int_a^b \Upsilon(t) d^{*}B(t) - \int_a^b \Phi_{m_n'}^{(n)}(t) d^{*}B(t) \right| > \varepsilon \right\} \\ \nonumber &\subset&
\left\{ \left| \int_a^b \Upsilon_n(t) d^{*}B(t) - \int_a^b \Phi_{m_n'}^{(n)}(t) d^{*}B(t) \right| > \frac{\varepsilon}{2} \right\} \\ \nonumber
&& \cup \left\{ \left| \int_a^b \Upsilon(t) d^{*}B(t) - \int_a^b \Upsilon_n(t) d^{*}B(t) \right| > \frac{\varepsilon}{2} \right\},
\end{eqnarray}
where we have chosen, for each $n \in \mathbb{N}$, a $m_n' \in \mathbb{N}$ large enough such that
$$
\mathbb{P} \left\{ \left| \int_a^b \Upsilon_n(t) d^{*}B(t) - \int_a^b \Phi_{m_n'}^{(n)}(t) d^{*}B(t) \right| > \frac{\varepsilon}{2} \right\} \le \frac{1}{n}
$$
and the same holds for all $m \ge m_n'$;
note that this is possible as a consequence of the Ayed-Kuo integrability of the constituents of the sequence $\lbrace \Upsilon_n(t) \rbrace _{n=1}^{\infty}$. This in turn implies
\begin{eqnarray}\nonumber
&& \mathbb{P} \left\{ \left| \int_a^b \Upsilon(t) d^{*}B(t) - \int_a^b \Phi_{m_n'}^{(n)}(t) d^{*}B(t) \right| > \varepsilon \right\} \\ \nonumber &\le&
\mathbb{P} \left\{ \left| \int_a^b \Upsilon(t) d^{*}B(t) - \int_a^b \Upsilon_n(t) d^{*}B(t) \right| > \frac{\varepsilon}{2} \right\} \\ \nonumber
&& + \mathbb{P} \left\{ \left| \int_a^b \Upsilon_n(t) d^{*}B(t) - \int_a^b \Phi_{m_n'}^{(n)}(t) d^{*}B(t) \right| > \frac{\varepsilon}{2} \right\} \\ \nonumber &\le&
\mathbb{P} \left\{ \left| \int_a^b \Upsilon(t) d^{*}B(t) - \int_a^b \Upsilon_n(t) d^{*}B(t) \right| > \frac{\varepsilon}{2} \right\} + \frac{1}{n},
\end{eqnarray}
and consequently
\begin{eqnarray}\nonumber
&& \lim_{n \to \infty} \mathbb{P} \left\{ \left| \int_a^b \Upsilon(t) d^{*}B(t) - \int_a^b \Phi_{m_n'}^{(n)}(t) d^{*}B(t) \right| > \varepsilon \right\} \\ \nonumber &\le& \lim_{n \to \infty} \mathbb{P} \left\{ \left| \int_a^b \Upsilon(t) d^{*}B(t) - \int_a^b \Upsilon_{n}(t) d^{*}B(t) \right| > \frac{\varepsilon}{2} \right\}
\\ \nonumber &=& 0,
\end{eqnarray}
for each $\varepsilon>0$
by assumption (b), what as a matter of fact implies condition (b) in Definition~\ref{extak}, and where we have defined
\begin{eqnarray}\nonumber
\int_a^b \Upsilon(t) d^{*}B(t) &:=& \lim_{n \to \infty} \int_a^b \Upsilon_n(t) d^{*}B(t),  \ \ \ \ \text{in probability},
\end{eqnarray}
which is a well defined random variable again by assumption (b).

Since almost sure convergence implies convergence in probability,
for each $n \in \mathbb{N}$ we may choose a $m_n \in \mathbb{N}$ large enough such that
\begin{eqnarray}\nonumber
\mathbb{P} \left\{ \int_a^b \left| \Upsilon_n(t) - \Phi_{m_n}^{(n)}(t) \right|^2 dt
> \frac{\delta}{4} \right\} \le \frac{1}{n}, \qquad \delta>0,
\end{eqnarray}
and the same holds for all $m \ge m_n$.
This is possible on account of $\lbrace \Upsilon_n(t) \rbrace _{n=1}^{\infty}$ being a family of Ayed-Kuo integrable processes,
and where $\lbrace \Phi_{m}^{(n)}(t) \rbrace _{m=1}^{\infty}$ is the same approximating sequence
as before.
Arguing similarly as before we find
\begin{eqnarray}\nonumber
&& \left\{ \int_a^b \left| \Upsilon(t) - \Phi_{m_n}^{(n)}(t) \right|^2 dt
> \delta \right\} \\ \nonumber &\subset&
\left\{ \int_a^b \left| \Upsilon(t) - \Upsilon_n(t) \right|^2 dt
> \frac{\delta}{4} \right\}
\cup \left\{\int_a^b \left| \Upsilon_n(t) - \Phi_{m_n}^{(n)}(t) \right|^2 dt  > \frac{\delta}{4} \right\},
\end{eqnarray}
and therefore
\begin{eqnarray}\nonumber
&& \mathbb{P} \left\{ \int_a^b \left| \Upsilon(t) - \Phi_{m_n}^{(n)}(t) \right|^2 dt > \delta \right\} \\ \nonumber &\le&
\mathbb{P} \left\{ \int_a^b \left| \Upsilon(t) - \Upsilon_n(t) \right|^2 dt > \frac{\delta}{4} \right\}
+ \mathbb{P} \left\{ \int_a^b \left| \Upsilon_n(t) - \Phi_{m_n}^{(n)}(t) \right|^2 dt > \frac{\delta}{4} \right\} \\ \nonumber &\le&
\mathbb{P} \left\{ \int_a^b \left| \Upsilon(t) - \Upsilon_n(t) \right|^2 dt > \frac{\delta}{4} \right\} + \frac{1}{n}.
\end{eqnarray}
In consequence
\begin{eqnarray}\nonumber
&& \lim_{n \to \infty} \mathbb{P} \left\{ \int_a^b \left| \Upsilon(t) - \Phi_{m_n}^{(n)}(t) \right|^2 dt > \delta \right\} \\ \nonumber &\le& \lim_{n \to \infty} \mathbb{P} \left\{ \int_a^b \left| \Upsilon(t) - \Upsilon_n(t) \right|^2 dt > \frac{\delta}{4} \right\}
\\ \nonumber &=& 0,
\end{eqnarray}
for each $\delta>0$ by assumption (a). Now, by taking a subsequence
$\left\lbrace \Phi_{m_{n_j}}^{(n_j)}(t) \right\rbrace _{j=1}^{\infty} \subset
\left\lbrace \Phi_{m_{n}}^{(n)}(t) \right\rbrace _{n=1}^{\infty}$
if necessary, we conclude
$$
\lim_{j \to \infty}
\int_a^b \left| \Upsilon(t) - \Phi_{m_{n_j}}^{(n_j)}(t) \right|^2 dt =0, \qquad \text{almost surely};
$$
therefore condition (a) in Definition~\ref{extak} follows.

Altogether these results imply
\begin{eqnarray}\nonumber
\int_a^b \Upsilon(t) d^{*}B(t) &=& \lim_{j \to \infty}
\int_a^b \Phi_{m_{n_j} \vee m_{n_j}'}^{(n_j)}(t) d^{*}B(t), \ \ \ \ \text{in probability},
\end{eqnarray}
and
$$
\int_a^b \left| \Upsilon(t) - \Phi_{m_{n_j} \vee m_{n_j}'}^{(n_j)}(t) \right|^2 dt \ \to \ 0
\ \ \ \  \mbox{almost surely}
$$
as $j \to \infty$,
so the statement follows.
\end{proof}

Now we are ready to prove existence and uniqueness of the solution to equation~\eqref{ak1}, along with an explicit representation formula.

\begin{theorem}\label{exunak}
The unique solution of~\eqref{ak1} and~\eqref{ak2} is
$$
S_1(t) = M f(B(T) - \sigma t) e^{\left(\mu - \sigma^2 /2 \right)t + \sigma B(t)}.
$$
\end{theorem}

\begin{proof}
Using the It\^o formula for the Ayed-Kuo integral~\cite{hksz} it is possible to solve problem~\eqref{ak1}-\eqref{ak2} to find
\begin{equation}\label{solfic}
S_1(t) = M \tilde{f}(B(T) - \sigma t) e^{\left(\mu - \sigma^2 /2 \right)t + \sigma B(t)},
\end{equation}
for $\tilde{f} \in \mathcal{C}(\mathbb{R})$; in particular, it is an Ayed-Kuo integrable process.
If we fix a realization of Brownian motion then, by \textit{Lusin Theorem}~\cite{lusin}, there exists a family of continuous functions $\{f_n\}_{n \in \mathbb{N}}$ such that
\begin{eqnarray}\label{lusin}
\int_{0}^{T} |f(B(T) - \sigma t) - f_n(B(T) - \sigma t)| dt \to 0,
\end{eqnarray}
with $||f_n||_\infty \le ||f||_\infty$, as $n \to \infty$. Then it holds that
\begin{eqnarray}\nonumber
&& \int_{0}^{T} |f(B(T) - \sigma t) - f_n(B(T) - \sigma t)|^2 dt \\ \nonumber
&\le& ||f(B(T) - \sigma t) - f_n(B(T) - \sigma t)||_\infty \\ \nonumber
&& \times \int_{0}^{T} |f(B(T) - \sigma t) - f_n(B(T) - \sigma t)| dt \\ \nonumber
&\le& \left[||f(B(T) - \sigma t)||_\infty + ||f_n(B(T) - \sigma t)||_\infty \right] \\ \nonumber
&& \times \int_{0}^{T} |f(B(T) - \sigma t) - f_n(B(T) - \sigma t)| dt \\ \nonumber
&\le& 2 ||f(B(T) - \sigma t)||_\infty
\int_{0}^{T} |f(B(T) - \sigma t) - f_n(B(T) - \sigma t)| dt \\ \nonumber
&\to& 0, \qquad \text{as} \quad n \to \infty;
\end{eqnarray}
therefore we conclude
\begin{equation}\label{convl2}
\int_{0}^{T} |f(B(T) - \sigma t) - f_n(B(T) - \sigma t)|^2 dt \to 0 \qquad \text{almost surely}
\end{equation}
as $n \to \infty$.

Note that equation~\eqref{ak1} means
\begin{equation}\label{eqmeans}
S_1(t) = S_1(0) + \mu \int_0^t S_1(s) \, ds + \sigma \int_0^t S_1(s) \, d^* B(s).
\end{equation}
As
$$
S_1^{(n)}(t) = M f_n(B(T) - \sigma t) e^{\left(\mu - \sigma^2 /2 \right)t + \sigma B(t)},
$$
is a solution to~\eqref{ak1} by~\eqref{solfic}, then
$$
S_1^{(n)}(t) = S_1^{(n)}(0) + \mu \int_0^t S_1^{(n)}(s) \, ds + \sigma \int_0^t S_1^{(n)}(s) \, d^* B(s)
$$
by~\eqref{eqmeans}. Note that
\begin{eqnarray}\nonumber
\int_0^t |S_1^{(n)}(s) - S_1(s)| \, ds &=& M \int_0^t |f_n(B(T) - \sigma s) - f(B(T) - \sigma s)|
\\ \nonumber
&& \times e^{\left(\mu - \sigma^2 /2 \right)s + \sigma B(s)} \, ds \\ \nonumber
&\le& M \max_{0 \le s \le t} \left\{
e^{\left(\mu - \sigma^2 /2 \right)s + \sigma B(s)} \right\}
\\ \nonumber
&& \times \int_0^t |f_n(B(T) - \sigma s) - f(B(T) - \sigma s)| \, ds \\ \label{convl1}
&\to& 0 \qquad \text{almost surely as} \quad n \to \infty
\end{eqnarray}
by~\eqref{lusin}. In consequence
\begin{equation} \label{subseq}
S_1^{(n_j)}(t) \to S_1(t)
\qquad \text{for almost every} \quad t \in [0,T]
\end{equation}
as $j \to \infty$ by passing, if necessary, to a suitable subsequence
$\left\{S_1^{(n_j)}\right\}_{j \in \mathbb{N}} \subset \left\{S_1^{(n)}\right\}_{n \in \mathbb{N}}$.
Now, by taking the limit $j \to \infty$, we find
\begin{equation}\label{fin1}
\sigma \lim_{j \to \infty} \int_0^t S_1^{(n_j)}(s) \, d^* B(s) = S_1(t) - S_1(0) - \mu \int_0^t S_1(s) \, ds
\end{equation}
for almost every $t \in [0,T]$ almost surely, and hence in probability, by~\eqref{convl1} and~\eqref{subseq}; note that
$$
S_1(0)= M f(B(T))
$$
is well defined almost surely. Therefore
condition (b) in Lemma~\ref{lemmappr} is met.
For condition (a) compute
\begin{eqnarray}\nonumber
\int_0^t |S_1^{(n)}(s) - S_1(s)|^2 \, ds &=& M^2 \int_0^t |f_n(B(T) - \sigma s) - f(B(T) - \sigma s)|^2
\\ \nonumber
&& \times e^{2\left(\mu - \sigma^2 /2 \right)s + 2\sigma B(s)} \, ds \\ \nonumber
&\le& M^2 \max_{0 \le s \le t} \left\{
e^{2\left(\mu - \sigma^2 /2 \right)s + 2\sigma B(s)} \right\}
\\ \nonumber
&& \times \int_0^t |f_n(B(T) - \sigma s) - f(B(T) - \sigma s)|^2 \, ds \\ \label{fin2}
&\to& 0 \qquad \text{almost surely as} \quad n \to \infty
\end{eqnarray}
by~\eqref{convl2}.
Now by Lemma~\ref{lemmappr}, \eqref{fin1}, and~\eqref{fin2} we conclude
$$
S_1(t) = S_1(0) + \mu \int_0^t S_1(s) \, ds + \sigma \int_0^t S_1(s) \, d^* B(s)
$$
for almost every $t \in [0,T]$ almost surely, and the existence and explicit representation part of the proof is finished.
For the uniqueness assume on the contrary that there exist two solutions $S_1'(t)$ and $S_1''(t)$ to find
$$
S_1^{(-)}(t) = \mu \int_0^t S_1^{(-)}(s) \, ds + \sigma \int_0^t S_1^{(-)}(s) \, d^* B(s),
$$
where $S_1^{(-)}(t) := S_1'(t) - S_1''(t)$. Since this equation has the unique solution
$S_1^{(-)}(t)=0$~\cite{hksz}, we have reached a contradiction.
\end{proof}

Finally, we can establish the optimal investment strategy for the insider under Ayed-Kuo integration. For this we assume as before the no-shorting condition $0 \leq f \leq 1$ and we employ the notation
$$
M^{\text{(AK)}}(T) := S_0(T) + S_1(T).
$$

\begin{theorem}\label{mainthak}
Let $f$ be a function of $B(T)$ such that $f \in L^\infty(\mathbb{R})$ and $0 \leq f \leq 1$. The optimal investment strategy under Ayed-Kuo integration is
\begin{eqnarray}\nonumber
f(B(T)) &=& 1,
\end{eqnarray}
for model~\eqref{rode1}-\eqref{rode2} and~\eqref{ak1}-\eqref{ak2}.
\end{theorem}
\begin{proof}
By Theorem~\ref{exunak} we can solve problem~\eqref{ak1}-\eqref{ak2} explicitly to find
\begin{subequations}
\begin{eqnarray}\nonumber
S_0(t) &=& M \left(1-f(B(T))\right) e^{\rho t} \\ \nonumber
S_1(t) &=& M f(B(T) - \sigma t) e^{\left(\mu - \sigma^2 /2 \right)t + \sigma B(t)}.
\end{eqnarray}
\end{subequations}
Our aim is to find the strategy $f$ such that $\mathbb{E}\left[M^{\text{(AK)}}(T)\right]$ is maximized. Hence, we have
\begin{eqnarray}\nonumber
\mathbb{E}\left[M^{\text{(AK)}}(T)\right] &=& \mathbb{E}\left[S_0(T)\right] + \mathbb{E}\left[S_1(T)\right] \\ \nonumber
&=& M \left( 1- \mathbb{E}\left[f(B(T))\right]\right) e^{\rho T} \\ \nonumber
& & + M \mathbb{E}\left[f(B(T) - \sigma T)e^{\sigma B(T)}\right] e^{\left(\mu - \sigma^2/2\right)T} \\ \nonumber
&=& M e^{\rho T} \left( 1 - \int_{-\infty}^{\infty} \frac{1}{\sqrt{2 \pi T}} f(x) e^{-\frac{x^2}{2T}}dx \right) \\ \nonumber
& & + Me^{\left(\mu - \sigma^2 /2 \right)T} \left( \int_{-\infty}^{\infty} \frac{1}{\sqrt{2 \pi T}} f(x-\sigma T) e^{-\frac{x^2}{2T}} e^{\sigma x} dx \right),
\end{eqnarray}
where we have used that $B(T) \sim \mathcal{N}(0,T)$. By changing variables
\begin{eqnarray}\nonumber
    y &=& x - \sigma T,
\end{eqnarray}
and
\begin{eqnarray}\nonumber
    \Bar{M}(T) &=& \frac{\mathbb{E}\left[M^{\text{(AK)}}(T)\right]}{M},
\end{eqnarray}
where the first change is implemented in the second integral only,
we find
\begin{eqnarray}\nonumber
\Bar{M}(T) &=& e^{\rho T} \left( 1 - \int_{-\infty}^{\infty} \frac{1}{\sqrt{2 \pi T}} f(x) e^{-\frac{x^2}{2T}}dx \right) \\ \nonumber
&& +e^{\left(\mu - \sigma^2 /2 \right)T} \left( \int_{-\infty}^{\infty} \frac{1}{\sqrt{2 \pi T}} f(x-\sigma T) e^{-\frac{x^2}{2T}} e^{\sigma x} dx \right) \\ \nonumber
&=& e^{\rho T} - e^{\rho T} \int_{-\infty}^{\infty} \frac{1}{\sqrt{2 \pi T}} f(x) e^{-\frac{x^2}{2T}} dx \\ \nonumber
&& +\int_{-\infty}^{\infty} \frac{1}{\sqrt{2 \pi T}} f(y) e^{-\frac{\left(y + \sigma T \right)^2}{2T}} e^{\left(\mu - \sigma ^2 /2 \right)T} e^{\sigma \left(y+\sigma T\right)}dy \\ \nonumber
&=& e^{\rho T} - e^{\rho T} \int_{-\infty}^{\infty} \frac{1}{\sqrt{2 \pi T}} f(x) e^{-\frac{x^2}{2T}} dx \\ \nonumber
&& +\int_{-\infty}^{\infty} \frac{1}{\sqrt{2 \pi T}} f(y) e^{-\frac{\left(y^2 + \sigma^2 T^2 +2y\sigma T \right)}{2T}} e^{\mu T - \sigma ^2 T /2} e^{\sigma y+\sigma^2 T}dy \\ \nonumber
&=& e^{\rho T} - e^{\rho T} \int_{-\infty}^{\infty} \frac{1}{\sqrt{2 \pi T}} f(x) e^{-\frac{x^2}{2T}} dx + e^{\mu T} \int_{-\infty}^{\infty} \frac{1}{\sqrt{2 \pi T}} f(y) e^{-\frac{y^2}{2T}} dy.
\end{eqnarray}
Therefore, we may summarize this as
\begin{eqnarray}\nonumber
    \Bar{M}(T) &=& e^{\rho T} - e^{\rho T} \mathbb{E}\left[f(B(T))\right] + e^{\mu T} \mathbb{E}\left[f(B(T))\right].
\end{eqnarray}
By assumption, $f$ is a function of $B(T)$ such that $0 \leq f \leq 1$, and $\mu > \rho$; so we have a convex linear combination and, since the exponential function is strictly monotone, we conclude $\mathbb{E}\left[M^{\text{(AK)}}(T)\right] \in [Me^{\rho T}, Me^{\mu T}]$. Then clearly $\mathbb{E}\left[M^{\text{(AK)}}(T)\right]$ is maximized whenever
\begin{eqnarray}\nonumber
    \mathbb{E}\left[f(B(T))\right] &=& 1.
\end{eqnarray}
Equivalently we have
\begin{eqnarray}\nonumber
    \frac{1}{\sqrt{2 \pi T}} \int_{-\infty}^{\infty} f(x) e^{-\frac{x^2}{2T}} dx &=& 1,
\end{eqnarray}
and consequently $f(B(T)) = 1$.
\end{proof}

\begin{remark}\label{maincolak}
The same conclusion found in Theorem~\ref{mainthak} can be reached by approximating $\mathbb{E}\left[f(B(T))\right]$
instead of $f$ (as done in Theorem~\ref{exunak}); let us show how.
We start considering $\{f_n\}_{n=1}^\infty$, a sequence of functions such that $f_n \in \mathcal{C}(\mathbb{R})$ for all $n \in \mathbb{N}$. Then, we have
\begin{eqnarray}\nonumber
    | \mathbb{E}\left[f(B(T))\right] - \mathbb{E}\left[f_n(B(T))\right] | &=& | \mathbb{E}\left[f(B(T)) - f_n(B(T)) \right] | \\ \nonumber
    &\leq& \mathbb{E}\left[| f(B(T)) - f_n (B(T)) | \right] \\ \nonumber
    &=& \frac{1}{\sqrt{2 \pi T}} \int_{-\infty}^{\infty} |f(x) - f_n(x)| e^{-\frac{x^2}{2T}} dx.
\end{eqnarray}
By \textit{Lusin Theorem}~\cite{lusin}, there exists a family $\{f_n\}_{n \in \mathbb{N}}$ of continuous functions,
$||f_n||_\infty \le ||f||_\infty$, such that
\begin{eqnarray}\nonumber
    \frac{1}{\sqrt{2 \pi T}} \int_{-\infty}^{\infty} |f(x) - f_n(x) | e^{-\frac{x^2}{2T}} dx \to 0
\end{eqnarray}
as $n \to \infty$, and where $0 \le f_n \le 1$ after a possible redefinition of their negative parts by zero,
so they constitute a well defined family of investment allocations under the no-shorting condition. In consequence
\begin{eqnarray}\nonumber
    \mathbb{E}\left[f_n(B(T))\right] &\to& \mathbb{E}\left[f(B(T))\right]
\end{eqnarray}
as $n \to \infty$. Since we found in the proof of Theorem~\ref{mainthak} that
\begin{eqnarray}\nonumber
    \mathbb{E}\left[M^{(AK)}(T)\right] &=& e^{\rho T} - e^{\rho T} \mathbb{E}\left[f(B(T))\right] + e^{\mu T} \mathbb{E}\left[f(B(T))\right],
\end{eqnarray}
we may conclude
\begin{eqnarray}\nonumber
    \mathbb{E}\left[M^{(AK)}(T)\right] &=& \lim_{n \to \infty}
    \left\{ e^{\rho T} - e^{\rho T} \mathbb{E}\left[f_n(B(T))\right] + e^{\mu T} \mathbb{E}\left[f_n(B(T))\right]
    \right\}.
\end{eqnarray}
So the same optimal investment is found by considering $f \in \mathcal{C}(\mathbb{R})$ in~\eqref{ak2} (so the existence, uniqueness, and explicit representation theory of~\cite{hksz} can be employed for the solution of~\eqref{ak1}) and then arguing by approximation as exposed in this Remark. Note anyway that this fact just illustrates the consistency of our approach overall, but the only complete proof is the one that goes through Theorem~\ref{exunak}.
\end{remark}

\begin{remark}
The function $f$ from Theorem~\ref{mainthak} suggests that the formalization of the problem based on the Ayed-Kuo integral does not take advantage of the anticipating initial condition, as this optimal investment strategy is the same as the one of the honest trader.
Indeed, it implies to invest the whole amount $M$ in the stock, in such a way that
\begin{eqnarray}\nonumber
    \mathbb{E}\left[M^{(AK)}(T)\right] &=& M e^{\mu T}.
\end{eqnarray}
Thus, the result of the use of the Ayed-Kuo integral seems to be counterintuitive in the financial sense, at least for this formulation of the insider trading problem.
\end{remark}

\section{The Hitsuda-Skorokhod integral}\label{hsi}

The Hitsuda-Skorokhod integral is an anticipating stochastic integral that was introduced by Hitsuda~\cite{hitsuda} and Skorokhod~\cite{skorokhod}
by means of different methods.
The following definition makes use of the Wiener-It\^o chaos expansion; background on this topic can be found for instance in~\cite{noep,hoeuz}.

\begin{definition}
Let $X \in L^2([0,T]\times \Omega)$ be a square integrable stochastic process. By the Wiener-It\^o chaos expansion, $X$ can be decomposed into an orthogonal series
$$X(t,\omega) = \sum_{n=0}^{\infty} I_n(f_{n,t}),$$
in $L^2(\Omega)$, where $f_{n,t}\in L^2([0,T]^n)$ are symmetric functions for all non-negative integers $n$. Thus, we write
$$ f_{n,t}(t_1,\ldots,t_n)=f_n(t_1,\ldots,t_n,t),$$
which is a function defined on $[0,T]^{n+1}$ and symmetric with respect to the first $n$ variables.
The symmetrization of $f_n(t_1,\ldots,t_n,t_{n+1})$ is given by
\begin{align*} \label{simetrizacion}
&\f_n(t_1,\ldots,t_{n+1})= \\
&\frac{1}{n+1}\left[f_n(t_1,\ldots,t_{n+1})+ f_n(t_{n+1},t_2,\ldots,t_1)+\ldots+f_n(t_1,\ldots,t_{n+1},t_{n})\right],
\end{align*}
because we only need to take into account the permutations which exchange the last variable with any other one.
Then, the Hitsuda-Skorokhod integral of $X$ is defined by
\begin{equation*}
\int_0^T X(t,\omega) \, \delta B(t) := \sum_{n=0}^{\infty} I_{n+1}(\f_n),
\end{equation*}
provided that the series converges in $L^2(\Omega)$.
\end{definition}

For the Hitsuda-Skorokhod integral we arrive at the initial value problem
\begin{subequations}
\begin{eqnarray}\label{sko1}
\delta S_1 &=& \mu \, S_1 \, d t + \sigma \, S_1 \, \delta B_t \\ \label{sko2}
S_1(0) &=& M f(B(T)),
\end{eqnarray}
\end{subequations}
for a Hitsuda-Skorokhod stochastic differential equation, where $\delta$ denotes the Hitsuda-Skorokhod stochastic differential.
The existence and uniqueness theory for linear stochastic differential equations of Hitsuda-Skorokhod type, which covers the present case,
can be found in~\cite{lssdes}; as before, $f$ is a function of $B(T)$, such that $f \in L^{\infty}(\mathbb{R})$ and $0 \leq f \leq 1$. In this case we have the following result.

\begin{theorem}\label{mainthhs}
The unique solution of~\eqref{sko1} and~\eqref{sko2} is
$$
S_1(t) = M f(B(T) - \sigma t) e^{\left(\mu - \sigma^2 /2 \right)t + \sigma B(t)}.
$$
\end{theorem}

\begin{proof}
The statement is a particular case of the existence and uniqueness result in~\cite{buckdahn}, which in turn states that the solution to stochastic differential equations of the form
\begin{equation*}\label{buckdahn1}
\delta Y_t = \mu_t \, Y_t \, dt + \sigma_t \, Y_t \, \delta B_t, \quad Y_0=\eta, \qquad 0 \leq t \leq T,
\end{equation*}
with $\sigma_t \in L^{\infty}([0,1])$, $\mu_t \in L^{\infty}([0,1]\times \Omega)$, and $ \eta  \in L^p(\Omega)$, $p > 2$ is given by
\begin{equation}\label{buckdahn2}
Y_t = \eta(U_{0,t}) \, \exp \left[ \int_0^t \mu_s(U_{s,t}) \, ds \right] X_t,  \ \ \ \
\mbox{almost surely},  \ \ \ \ 0 \leq t \leq T,
\end{equation}
where
$$
X_t = \exp \left( \int_0^t \sigma_s \, \delta B_s - \frac{1}{2} \int_0^t \sigma_s^2 \, ds \right), \qquad 0 \leq t \leq T,
$$
and, for $0\leq s \leq t \leq T$, the Girsanov transformation is
$$
U_{s,t} = B(u) - \int_0^u \1_{[s,t]}(r) \, \sigma_r \, dr, \ \ \ \ 0 \leq u \leq T.
$$
Now taking $\sigma_s=\sigma$ and $\mu_s=\mu$ as two constant processes, and $\eta=M f(B(T))$,
the different factors in $\eqref{buckdahn2}$ become
\begin{eqnarray}\nonumber
e^{\mu t} &=& \exp\left[ \int_0^t \mu_s(U_{s,t}) ds \right], \\ \nonumber
X_t &=& \exp\left[\sigma B(t) - \sigma^2 t /2 \right], \\ \nonumber
\eta(U_{0,t}) &=& M f(B(T)-\sigma t).
\end{eqnarray}
Since
$$
\mathbb{E} \left( \left|\eta\right|^p \right) \le \mathbb{E} (M^p) = M^p < \infty,
$$
the result follows.
\end{proof}

\begin{corollary}
Let $f$ be a function of $B(T)$ such that $f \in L^{\infty}(\mathbb{R})$ and $0 \leq f \leq 1$. The optimal investment strategy under Hitsuda-Skorokhod integration is
\begin{eqnarray}\nonumber
f(B(T)) &=& 1,
\end{eqnarray}
for model~\eqref{rode1}-\eqref{rode2} and~\eqref{sko1}-\eqref{sko2}.
\end{corollary}

\begin{proof}
The statement is a direct consequence of Theorem~\ref{mainthhs} and the proof of Theorem~\ref{mainthak}.
\end{proof}

\begin{remark}
Note that this result is the same as the one found in the previous section for the Ayed-Kuo integral,
that is
\begin{eqnarray}\nonumber
    \mathbb{E}\left[M^{(HS)}(T)\right] &=& M e^{\mu T}.
\end{eqnarray}
Therefore the same financial conclusions hold in this case as well.
\end{remark}

\section{Further results and comments}\label{frc}

In this section we address some questions that complement our previous developments.
First, let us consider the maximization of $\mathbb{E}\left[M(T) | B(T) \right]$ rather than the maximization of $\mathbb{E}\left[M(T)\right]$. This quantity is no longer a real number, but a random variable, so the
optimization problem can only be realized by means of the introduction of a notion of stochastic ordering.
Herein we will assume stochastic orderings given by different orders of stochastic dominance~\cite{levy}; although
other notions of stochastic order are possible~\cite{ss}, we base our decision on the fact that first order stochastic dominance
is the usual stochastic order.
The simplest situation is the one corresponding to the Russo-Vallois forward integral, that is,
to the problem analyzed in section~\ref{rvi};
we remind that $f$ is a function of $B(T)$ such that $f \in L^{\infty}(\mathbb{R})$ and $0 \leq f \leq 1$.
In this case, the conditional expectation of the insider wealth is
\begin{eqnarray}\nonumber
\mathbb{E}\left[M^{\text{(RV)}}(T) | B(T)\right]
&=& M \left[ \left( 1- f(B(T)) \right) e^{\rho T} + f(B(T)) e^{\left(\mu - \sigma^2/2\right)T + \sigma B(T)} \right]
\end{eqnarray}
almost surely, since $M^{\text{(RV)}}(T)$ is measurable by the sigma field generated by $B(T)$.
The last expression is nothing but a convex linear combination and, since the exponential function is strictly monotone,
we conclude that the optimal investment strategy is
\begin{eqnarray}\nonumber
    f(B(T)) &=& \mathlarger{\mathlarger{\mathbbm{1}}}_{\big\lbrace B(T) > \frac{T}{\sigma} \left( \rho - \mu + \frac{\sigma^2}{2} \right ) \big\rbrace}
\end{eqnarray}
almost surely. This maximizer coincides with the one provided by Theorem~\ref{mainthrv}.
Note that the notion of stochastic order employed is zeroth order stochastic dominance, which is the strongest notion
of stochastic dominance and implies all the higher orders~\cite{levy}.

In the case of either the Ayed-Kuo or the Hitsuda-Skorokhod integral,
see sections~\ref{aki} and~\ref{hsi} respectively, we find
\begin{eqnarray}\nonumber
\mathbb{E}\left[M^{\text{(AK/HS)}}(T) | B(T)\right] &=&
M \left[ \left( 1- f(B(T)) \right) e^{\rho T} \right. \\ \nonumber
&& \left. + f(B(T)-\sigma T) e^{\left(\mu - \sigma^2/2\right)T + \sigma B(T)} \right]
\end{eqnarray}
almost surely, since $M^{\text{(AK/HS)}}(T)$ is measurable by the sigma field generated by $B(T)$.
Since $\mathbb{E}\left[M^{\text{(AK/HS)}}(T)\right]$ is maximized for $f \equiv 1$, this should be the maximizer
in this case too, of course provided a maximizer exists~\cite{levy}. Note however that
\begin{eqnarray}\nonumber
\mathbb{E} \left[ \left. M^{\text{(AK/HS)}}(T) \right|_{f \equiv 1} \right] &=& e^{\mu T} \\ \nonumber
&>& e^{\rho T} \\ \nonumber
&=& \mathbb{E} \left[ \left. M^{\text{(AK/HS)}}(T) \right|_{f \equiv 0} \right],
\end{eqnarray}
by the assumption $\mu > \rho$ and the monotonicity of the exponential; but nevertheless
\begin{eqnarray}\nonumber
\inf \left\{ \left. M^{\text{(AK/HS)}}(T) \right|_{f \equiv 1} \right\} &=& 0 \\ \nonumber
&<& e^{\rho T} \\ \nonumber
&=& \min \left\{ \left. M^{\text{(AK/HS)}}(T) \right|_{f \equiv 0} \right\}.
\end{eqnarray}
These two results combined imply that $\left. M^{\text{(AK/HS)}}(T) \right|_{f \equiv 1}$
does not stochastically dominate $\left. M^{\text{(AK/HS)}}(T) \right|_{f \equiv 0}$
in the $m^{\text{th}}-$order for any $m=0,1,2,3,\cdots$~\cite{levy}. In other words, assuming a stochastic ordering
given by the $m^{\text{th}}-$order stochastic dominance, there exists no optimal value for any $m=0,1,2,3,\cdots$.
On one hand, this should not be regarded as a bizarre outcome since stochastic orders are in general partial orders.
On the other hand, it seems possible to turn $f \equiv 1$ into the maximizer of the problem by means of the introduction
of a suitable risk-seeking stochastic order. In any case, the obvious difference between the results presented in this section,
along with the fact that
$$
\mathbb{E} \left[ \left. M^{\text{(AK/HS)}}(T) \right|_{f \equiv 1} \right] <
\mathbb{E} \left[ \left. M^{\text{(RV)}}(T) \right|_{f=
\mathlarger{\mathlarger{\mathbbm{1}}}_{\big\lbrace B(T) > \frac{T}{\sigma} \left( \rho - \mu + \frac{\sigma^2}{2} \right )
\big\rbrace}} \right],
$$
both point in the same direction of support of the Russo-Vallois forward integral,
at least under the particular conditions we are considering herein.

It is also important to clarify that the insider problem we have addressed is solvable by means of much simpler methods that
do not require the introduction of anticipating stochastic integration; a fact that
we anticipated in section~\ref{insider}.
To see this consider our basic model subject to unit initial conditions, that is
\begin{subequations}
\begin{eqnarray}\label{fr0}
d S_0 &=& \rho \, S_0 \, d t \\ \label{fr1}
S_0(0) &=& 1,
\end{eqnarray}
\end{subequations}
and
\begin{subequations}
\begin{eqnarray}\label{fs0}
d S_1 &=& \mu \, S_1 \, d t + \sigma \, S_1 \, d B(t) \\ \label{fs1}
S_1(0) &=& 1.
\end{eqnarray}
\end{subequations}
Clearly~\eqref{fs0}-\eqref{fs1} is an It\^o stochastic differential equation and~\eqref{fr0}-\eqref{fr1} is an ordinary differential equation. With these two assets one can build the portfolio
\begin{eqnarray}\nonumber
M(t) &=& M_0 S_0(t) + M_1 S_1(t)
\\ \nonumber
&=& M \left(1-f\right) S_0(t) + M f S_1(t) \\ \label{solito}
&=& M \left[ \left( 1- f \right) e^{\rho T} + f e^{\left(\mu - \sigma^2/2\right)T + \sigma B(T)} \right],
\end{eqnarray}
where we have assumed that the total initial wealth $M=M_0 + M_1$ is constant and $f$ denotes the fraction of the total wealth
invested in the stock. The computation of the expected final wealth depends on whether $f$ is a constant or
we allow $f=f(B(T))$. In the first case we find
\begin{eqnarray}\nonumber
\mathbb{E}\left[M(T)\right]
&=& M\left(1-f\right) \, e^{\rho T} + M \, f \, e^{\mu T},
\end{eqnarray}
that is, our classical convex linear combination, which gets maximized for
\begin{eqnarray}\nonumber
    f &=& 1.
\end{eqnarray}
Of course, we find agreement with the corresponding result in section~\ref{introduction}.
In the latter case we get
\begin{eqnarray}\nonumber
\mathbb{E}\left[M(T)\right] &=& M \left( 1- \mathbb{E}\left[f(B(T))\right] \right) e^{\rho T} +
M \mathbb{E}\left[ f(B(T)) e^{\sigma B(T)} \right] e^{\left(\mu - \sigma^2/2\right)T},
\end{eqnarray}
in agreement with our development based on the Russo-Vallois integral, see section~\ref{rvi};
as there we conclude the optimizer is
\begin{eqnarray}\label{maxito}
    f(B(T)) &=& \mathlarger{\mathlarger{\mathbbm{1}}}_{\big\lbrace B(T) > \frac{T}{\sigma} \left( \rho - \mu + \frac{\sigma^2}{2} \right ) \big\rbrace}.
\end{eqnarray}
This computation allows us to highlight two things. First, we find yet another argument in favor of the forward integral (as always, in our limited setting). Second, it is a reminder of the fact that our goal is to establish a comparison, as clear as possible, between the
use of different anticipating stochastic integrals in a financial context. To this end we need a simple enough problem, which solution
is not the objective of this work.

Finally, let us remark that our selection of the possible interpretations of noise is not exhaustive.
One could consider, for instance, pathwise integration theories such as the one by F\"ollmer~\cite{follmer}
or Rough Path Theory (with It\^o enhanced Brownian motion in our case)~\cite{fh}. Since these theories replicate
It\^o calculus, their use will presumably
lead to the same result as the Russo-Vallois forward integral. Moreover,
these theories are able to replicate neither Ayed-Kuo nor Hitsuda-Skorokhod calculus~\cite{fh}.
On the other hand, expressions such as~\eqref{solito} make appealing the use of a pathwise theory in the
context of the present problem since, as we have discussed at the beginning of this section, the strategy
given by~\eqref{maxito} maximizes the wealth of the insider not just in mean but almost surely. At this point,
however, it is convenient to realize again that the problem under consideration is a simple model which permits
a full comparison between the three anticipating integrals. In more complex problems, for instance with partial
insider information, it seems to be too optimistic to hope for almost surely optimal strategies.
Note that an akin situation arises in the case treated in the second paragraph of this section. In such general cases,
in which some type of average or suitable stochastic ordering seems to be needed, it is perhaps too restrictive
to ask for a purely analytical solution to the optimization problem.

\section{Conclusions}
In this work we have considered a version of insider trading that allows us to compare the optimal investment strategies provided by three different anticipating stochastic integrals: the Russo-Vallois forward, the Ayed-Kuo and the Hitsuda-Skorokhod integrals.

Specifically, we have considered the following formulation of insider trading for the bank account
\begin{subequations}
\begin{eqnarray*}\label{rode61}
d S_0 &=& \rho \, S_0 \, d t \\ \label{rode62}
S_0(0) &=& M \left(1-f(B(T))\right),
\end{eqnarray*}
\end{subequations}
and for the stock
\begin{subequations}
\begin{eqnarray*}\label{s61}
\dj S_1 &=& \mu \, S_1 \, d t + \sigma \, S_1 \, \dj B(t) \\ \label{s610}
S_1(0) &=& M f(B(T)),
\end{eqnarray*}
\end{subequations}
where the anticipating initial condition $f$ is a function of $B(T)$, such that $f \in L^{\infty}(\mathbb{R})$ and $0 \leq f \leq 1$,
and $\dj \in \{d^-,d^*,\delta\}$ is a stochastic differential of one of the types under consideration. That is, we have only considered
buy-and-hold strategies for which shorting is not allowed.

Our main task has been to establish the optimal investment strategy for each of the anticipating integrals. When the choice is the Russo-Vallois integral, the investment strategy that maximizes the expected wealth is to invest the whole amount $M$ in the asset whose actual value is larger at maturity time. Thus,
\begin{eqnarray}\nonumber
    f(B(T)) &=& \mathlarger{\mathlarger{\mathbbm{1}}}_{\big \lbrace B(T) > \frac{T}{\sigma} \left( \rho - \mu + \frac{\sigma^2}{2} \right) \big \rbrace }.
\end{eqnarray}
For the Ayed-Kuo and the Hitsuda-Skorokhod integrals, the optimal investment strategy is the same as the one in the It\^o case (that is, the uninformed case), which implies to invest the whole amount $M$ in the stock. Hence,
\begin{eqnarray}\nonumber
    f(B(T)) &=& 1.
\end{eqnarray}
These results suggest that the Russo-Vallois integral works as one would expect from a financial perspective, while the Ayed-Kuo and Hitsuda-Skorokhod integrals provide a solution that seems to be counterintuitive in the financial sense. Indeed, from our present results along with those in~\cite{escudero} it follows that
\begin{equation}\nonumber
\mathbb{E}\left[S^{\text{(I)}}(T)\right] = \mathbb{E}\left[M^{\text{(AK)}}(T)\right] = \mathbb{E}\left[M^{\text{(HS)}}(T)\right] < \mathbb{E}\left[M^{\text{(RV)}}(T)\right],
\end{equation}
where we have chosen the optimal investment allocation in each case. If we selected the optimal Russo-Vallois strategy for all the three anticipating cases the situation becomes even worse, see~\cite{escudero}; and it could be even more critical if we allowed different strategies, see~\cite{bastonsescudero}. All in all, our results suggest that while the Russo-Vallois forward integral allows the insider to make full use of the privileged information, both the Ayed-Kuo and Hitsuda-Skorokhod integrals effectively transform the insider into an uninformed trader.

Let us finish mentioning that it is also notorious that the Ayed-Kuo and Hitsuda-Skorokhod integrals give
simpler mathematical results. This is an indication, at least to us, of their potential use in different applications, be them financial or other. As the more classical versions of the interpretation of
noise problem show, it is not easy to anticipate what they could be, since the right interpretation
should in general be chosen on a problem by problem basis~\cite{escudero2}.

\section*{Acknowledgments}
We are grateful to two anonymous referees for their insightful comments and suggestions.

\bibliographystyle{amsplain}

\end{document}